\documentclass[11pt]{article}
\usepackage{amsmath,amssymb}
\usepackage{hyperref}
\usepackage{amsthm}

\newtheorem{Theorem}{Theorem}[section]

\newtheorem{Definition}[Theorem]{Definition}
\newtheorem{Proposition}[Theorem]{Proposition}

\newtheorem{Remark}[Theorem]{Remark}

\begin{document}
\title{On Characterizations of Some General $(\alpha,\,\beta)$ Norms in a Minkowski Space.
 }
\author{Li Yan\\
School of Mathematical Sciences\\ Peking  University,
Beijing 100871, P.R.China\\
E-mail:  liyanmath@pku.edu.cn}
\maketitle
\begin{abstract}
General $(\alpha,\,\beta)$ norms are an important class of Minkowski norms which contains the original $(\alpha,\,\beta)$ norms. In this note, by studying the behavior of the Darboux curves (see Definition \ref{Darboux} below) of the indicatrix, we give a characterization of 3-dimensional general $(\alpha,\,\beta)$ norms. By studying the isoperimetric properties of the indicatrix, as well as the isoperimetric inequalities in a Minkowski space, we give some global geometric quantities which characterizes Randers norms of arbitrary dimensions.\\
\textbf{Key words and phrases}: General $(\alpha,\,\beta)$ norm; Randers norm; Minkowski norm; Blaschke structure.\\
\textbf{Mathematics Subject Classification (2010)}: 53A15, 54C45, 53C60.
\end{abstract}
\section{Introduction}
\label{sec:1}
Recent years, the study of Finsler geometry has attracted a lot of attention, for basic and advanced topics, see \cite{GTM200}, \cite{Chern and Shen 2005}, \cite{Shenyibing}, etc. Methods from other fields of mathematics also have some intersting applications in Finsler geometry. For example, the Hamiltonian systems (cf. \cite{Bangert}, \cite{Wang}), the ergodic theory (cf. \cite{Anosov}, \cite{Katok}), etc. Roughly speaking, a Finsler metric $F$ on a manifold $M^n$ is a collection of Minkowski norms on each tangent space $T_xM^n,\, x\in M^n$, which varies smoothly on $M^n$. The model of the tangent space of a Finsler manifold is a vector space equipped with a Minkowski norm, namely, a Minkowski space.

Euclidean norms are of course Minkowski norms. Geometers first characterized Euclidean norms among the Minkowski ones. In 1953, Deicke \cite{Deicke} proved his famous theorem which states that a Minkowski norm on a vector space $V$ is Euclidean, if and only if its Cartan form vanishes. Later in 1964, Su gave another characterization by studying curves on the indicatrix:
\begin{Theorem} (\cite{Su Minkowski})
Let $V$ be a vector space and O its origin. Let $F$ be a Minkowski norm on it and $M$ the indicatrix of $F$. Then $F$ is a Euclidean norm if and only if: every geodesic of $M$ with respect to the metric induced by $F$ lies on a plane passing through O.
\end{Theorem}

Among the non-Euclidean Minkowski norms, Randers norm seems to be the simplest one, it can be represented as $F=\alpha+\beta$. Where $\alpha$ is a Euclidean norm and $\beta$ is an 1-form (cf. \cite{GTM200}, \cite{Randers}). Matsumoto-Hojo gave a characterization of $n(\geq3)$-dimensional Randers norms in \cite{Matsumoto 1972} and \cite{Matsumoto 1978}:
\begin{Theorem}\label{Thm12} (\cite{Matsumoto 1978}) Let $F$ be a Minkowski norm on an $n(\geq3)$-dimensional vector space, $F$ is a Randers norm if and only if its Matsumoto torsion (see (\ref{eq28}) below) vanishes.
\end{Theorem}
On the other hand, a Randers norm can also be viewed as a shift metric, i.e. a solution of the Zermelo's navigation problem with navigation data $(h,W)$ (see (\ref{eq29}) below). Where $h$ is a Euclidean norm and $W$ a vector satisfying $h(W)<1$ (cf. Chern-Shen \cite{Chern and Shen 2005}, pp.24). From this point of view, one can see that the indicatrix of a Randers norm can be derived by shifting the indicatrix of a Euclidean norm. Hence an $n$-dimensional Minkowski norm is a Rander norm if and only if its indicatrix is an $(n-1)$-dimensional ellipsoid. This is true for all $n\in \mathbb{N}$. Based on those observations, Mo-Huang proved in \cite{Mo-Huang 2010} that
\begin{Theorem}\label{Thm13} (\cite{Mo-Huang 2010}) Let $F$ be a Minkowski norm on an $n$-dimensional vector space, $F$ is a Randers norm if and only if $L_1$ (see (\ref{eq217}) below) is a constant along the indicatrix.
\end{Theorem}
Mo-Huang further pointed out that the Matsumoto torsion (possibly multiplied by a positive factor)is just the cubic form of the indicatrix with its Blaschke structure (see (\ref{eq215}) below). Hence the Matsumoto-Hojo Theorem (Theorem \ref{Thm12}) is a corollary of the Pick-Berwald Theorem (cf. \cite{Noumizu-Sasaki}, pp.53).

A more genral class of non-Euclidean norm is the $(\alpha,\,\beta)$ norm. A $(\alpha,\,\beta)$ norm is defined by $F=\alpha\phi\left({\frac{\beta}{\alpha}}\right)$. Where $\alpha$ is a Euclidean norm, $\beta$ is a 1-form and $\phi$ a smooth function satisfying certain conditions(cf. \cite{Chern and Shen 2005}, pp.5-7). In virtue of the navigation problem, Yu-Zhu in \cite{Yu-Zhu DGA} investigated the so-called general $(\alpha,\,\beta)$ norm, which can be viewed as a solution of (\ref{eq29}), where $F$ is an $(\alpha,\,\beta)$ norm and $U$ is a vector satisfies $F(U)<1$. They mentioned that in an $\alpha$-orthogonal coordinate, the indicatrix of an $n$-dimensional general $(\alpha,\,\beta)$ norm is an $(n-1)$-dimensional Euclidean hypersurface of revolution.

In this work, we will build up characterizations of some general $(\alpha,\,\beta)$ norms. We mainly use methods arising from affine differential geometry. We shall first review some basic affine properties of the indicatrix. In section 3, we will derive a characterization of 3-dimensional general $(\alpha,\,\beta)$ norms (Theorem \ref{Thm63}). We will show that the indicatrix of a 3-dimensional general $(\alpha,\,\beta)$ norm is an affine surface of revolution (Proposition \ref{prop61}) and vice versa. Based on Su's Theorem \ref{Thm62}, we will then find the characterization by studying the Darboux curves of the indicatrix.  In section 4, we will give new characterizations of Randers norms by proving a maximum property of Randers norms (see Theorem \ref{Thm52}), and also some integral inequalities on the indicatrix (see Theorem \ref{Thm53}-\ref{Thm54}). All the characterizations given in section 4 are global quantities of the indicatrix.

For convenience, we will mainly deal with in this paper an $(n+1)$-dimensional Minkowski space $(V^{n+1},\,F)$ equipped with its Busemann-Hausdorff volume form, since the results we obtained can be generalized to $(V^{n+1},\,F)$ equipped with other kinds of volume forms with no essential differences.

\section{Preliminaries}
\label{sec:2}
\subsection{The Minkowski norm}
\label{sec:21}
Let $V^{n+1}$ be an $(n+1)$-dimensional vecter space and $F$ a Minkoski norm on it. Denote by $M^n$ the indicatrix of $F$, i.e.
\begin{equation}
M^n:=\left\{y\in V^{n+1}\mid F(y)=1\right\}.
\end{equation}
It is well known that $M^n$ is an $n$-dimensional strictly convex hypersurface enclosing the origin O (cf. \cite{Bryant 2002}, \cite{Jeanne-Chri}). In a chosen coordinate of $V^{n+1}$, $M^n$ is uniquely determined by $F$ and vice versa (cf. \cite{GTM200}, \cite{Chern and Shen 2005}). Since $M^n$ is diffeomorphic to the standard unit sphere $S^{n}(1)$, we can parameterize $M^n$ by
\begin{equation}\label{eqpara}
\begin{aligned}
r&:\, S^n(1)\to M^n\left(\subset\ V^{n+1}\right)\\
r&=r\left(\theta^1,...,\theta^n\right)\\
&=\left(r^1,...,r^{n+1}\right),
\end{aligned}
\end{equation}
where $\theta^1,...,\theta^n$ are the $n$ angles of spherical coordinate system in $\mathbb{R}^{n+1}$. It follows that
\begin{equation}
r^A=r^A\left(\theta^1,...,\theta^n\right).
\end{equation}
Also, the unit Finsler ball with respect to $F$ is
\begin{equation}
B_F(1):=\left\{y\in V^{n+1}\mid F(y)<1\right\},
\end{equation}
which is a strictly convex domain in $\mathbb{R}^{n+1}$ with $\partial B_F(1)=M^n$.
Through out the paper, for $V^{n+1}$, let the capitalized Latin indices $A,\,B,\,C...$ run from $1$ to $n+1$, and the uncapitalized Latin indices $i,\,j,\,k...$ run from $1$ to $n$. For an arbitrary fixed affine coordinate $\{y^A\}_{A=1}^{n+1}$ in $V^{n+1}$, the fundamental tensor of $F$ is given by
\begin{equation}\label{eq21}
g_{AB}(y):={\frac{1}{2}}{\frac{\partial^2F^2}{\partial y^A \partial y^B}}(y).
\end{equation}
Differentiating $F^2(y)$ three times gives the Cartan torsion
\begin{equation}\label{eq22}
\mathbf A_{ABC}(y):={\frac{1}{4}}F{\frac{\partial^3F^2}{\partial y^A \partial y^B \partial y^C}}(y).
\end{equation}
The angular form of $F$ is
\begin{equation}\label{eq23}
h_{AB}(y):=g_{AB}(y)-{\frac{\partial F}{\partial y^A}}(y){\frac{\partial F}{\partial y^B}}(y).
\end{equation}
One can check that the restriction of $g(y)$ on $M^n$ is precisely $h(y)$. The Busemann-Hausdorff volume form, is given by
\begin{equation}\label{eq24}
dV_{B-H}:=\sigma_Fdy^1\wedge...\wedge dy^{n+1},
\end{equation}
where
\begin{equation}\label{eq25}
\sigma_F:={\frac{\omega_{n+1}}{Vol_{\mathbb{R}^{n+1}}\left(B_F(1)\right)}}.
\end{equation}
Here in (\ref{eq25}), we denote
\begin{equation}
Vol_{\mathbb{R}^{n+1}}\left(\Omega\right):=\int_{\Omega}1dy^1\wedge...\wedge y^{n+1}
\end{equation}
as the volume of a measurable set $\Omega\subset{\mathbb{R}^{n+1}}$, and $\omega_{n+1}={\frac{\pi^{\frac{n+1}{2}}}{\Gamma\left({\frac{n+3}{2}}\right)}}$ is the volume of the standard $(n+1)$-dimensional Euclidean unit ball. The distorsion $\tau$ of $F$ with respect to $dV_{B-H}$ is defined by
\begin{equation}\label{eq26}
\tau(y):=\log{\frac{\sqrt{\det\left(g_{AB}(y)\right)}}{\sigma_F}}.
\end{equation}
It is easy to check that the Cartan form
\begin{equation}\label{eq27}
\mathbf I(y):=\mathbf I_A(y)dy^A:=g^{BC}(y)\mathbf A_{ABC}(y)dy^A
\end{equation}
equals to $d\tau$, where $\left(g^{AB}\right)=\left(g_{AB}\right)^{-1}$.
The Matsumoto torsion defined in \cite{Matsumoto 1972} and \cite{Matsumoto 1978} is
\begin{equation}\label{eq28}
\mathbf M_{ABC}(y):=\mathbf A_{ABC}(y)-{\frac{1}{n+3}}\left(\mathbf I_A(y)h_{BC}(y)+\mathbf I_B(y)h_{CA}(y)+\mathbf I_C(y)h_{AB}(y)\right).
\end{equation}
Throughout this paper, by saying $\tilde F(y)$ is a Minkowski norm with navigation data $(F,U)$, we mean that $\tilde F(y)$ is the solution of the following navigation problem (cf. \cite{Mo and Huang 2007}):
\begin{equation}\label{eq29}
F\left({\frac{y}{\tilde F(y)}}+U\right)=1,
\end{equation}
where $U$ is a vector in $V^{n+1}$ satisfies $F(U)<1$. It is proved that $\tilde F$ is uniquely determined by (\ref{eq29}) and is a Minkowski norm on $V^{n+1}$ whenever $F(U)<1$ is held (cf. \cite{Chern and Shen 2005}, \cite{Mo and Huang 2007}). And $\tilde M^n$, the indicatrix of $\tilde F$, satisfies
\begin{equation}\label{eq210}
\tilde M^n=M^n-U
\end{equation}
as a point set.
\subsection{Centro-affine structure of the indicatrix}
\label{sec:22}
Throughout the paper, vectors are considered as column vectors. For $n+1$ vectors $\{v_A\}_{A=1}^{n+1}$ with $v_A=v_A^B{\frac{\partial}{\partial y^B}}$, we write simply $\left(v_1,...,v_{n+1}\right)$ in short of the $(n+1)\times(n+1)$ matrix
\begin{equation}
\left(v_A^B\right)=\left(
\begin{aligned}
&v_1^1,\,&...\,,&v_n^1\\
&...\,&...\,&...\\
&v_n^1,\,&...\,,&v_n^n\\
\end{aligned}
\right).
\end{equation}

Choose $\mathbf{L}=-y$ as the transversal field on $M^n$ and take the parametrization (\ref{eqpara}) of $M^n$, one gets the centro-affine immersion of $M^n$ into $V^{n+1}$ (\cite{Bryant 2002}, \cite{Mo-Huang 2010}, \cite{Noumizu-Sasaki}), the centro-affine fundamental form coincides precisely with the angular form:
\begin{equation}\label{eq211}
h=h_{ij}d\theta^i\otimes d\theta^j,\quad h_{ij}=h_{AB}{\frac{\partial r^A}{\partial \theta^i}}{\frac{\partial r^B}{\partial \theta^j}}.
\end{equation}
And by definition (cf. \cite{Noumizu-Sasaki} \S{II.1}) we have
\begin{equation}\label{59plus}
h_{ij}={\frac{-\det\left({\frac{\partial r}{\partial \theta^1}},...,{\frac{\partial r}{\partial \theta^n}},D_{{\frac{\partial r}{\partial \theta^i}}}{\frac{\partial r}{\partial \theta^j}}\right)}{\det\left({\frac{\partial r}{\partial \theta^1}},...,{\frac{\partial r}{\partial \theta^n}},r\right)}}.
\end{equation}
The coefficients of the centro-affine connection $\nabla^{(c)}$ on $TM^n$ is given by
\begin{equation}\label{eq212}
\Gamma^{(c)k}_{ij}={\frac{1}{2}}h^{kl}\left({\frac{\partial h_{lj}}{\partial \theta^i}}+{\frac{\partial h_{il}}{\partial \theta^j}}-{\frac{\partial h_{ij}}{\partial \theta^l}}\right)-{\frac{1}{2}}h^{kl}\mathbf A_{ijl},
\end{equation}
where $\mathbf A_{ijl}=\mathbf A_{ABC}{\frac{\partial r^A}{\partial \theta^i}}{\frac{\partial r^B}{\partial \theta^j}}{\frac{\partial r^C}{\partial \theta^l}}$ and $\left(h^{ij}\right)=\left(h_{ij}\right)^{-1}$. And the corresponding cubic form $\nabla^{(c)} h$ is just the Cartan torsion $\mathbf A(\theta)=\mathbf A_{ijk}\left(r(\theta)\right)d\theta^i\otimes d\theta^j\otimes d\theta^k$, which is fully symmetric in $i,\,j,\,k$. For the centro-affine immersion of $M^n$, the shape operator is always $s^{(c)}=Id$ on $TM^n$. By (\ref{eq211}) and the fact that $g_{AB}y^Ay^B=F^2(y)$, we have for $\forall y\left(=r(\theta)\right)\in M^n$
\begin{equation}\label{eq311}
\left(\begin{aligned}(&h_{ij})_{n\times n}&0\\&0&1\end{aligned}\right)=
\left({\frac{\partial r}{\partial \theta^1}},...,{\frac{\partial r}{\partial \theta^n}},r\right)^T\left(g_{AB}\right)\left({\frac{\partial r}{\partial \theta^1}},...,{\frac{\partial r}{\partial \theta^n}},r\right).
\end{equation}
\subsection{The equiaffine structure of the indicatrix}
\label{sec:23}
Fix a volume form on the ambient affine space $V^{n+1}$, one can define the corresponding equiaffine structure (or Blaschke structure in other words) of a immersed hypersurface. For basic properties of the equiaffine structure of hypersurfaces, one can refer to Blaschke \cite{Blaschke ADG}, Su \cite{Su ADG} \S{I.5}, Nomizu-Sasaki \cite{Noumizu-Sasaki},etc. In the following, we just list some basic equiaffine properties of the indicatrix (cf. \cite{Mo-Huang 2010}).

The Blaschke metric of $M^n$ is given by
\begin{equation}\label{eq213}
G=G_{ij}d\theta^i\otimes d\theta^j,\,G_{ij}=\left[{\frac{\sigma_F^2}{\det\left(g_{AB}\right)}}\right]^{\frac{1}{n+2}}h_{ij}.
\end{equation}
While the affine norm field is given by
\begin{equation}\label{eq214}
\begin{aligned}
\xi(y):&={\frac{1}{n}}\triangle_{G}y\\
&=-\left[{\frac{\sigma_F^2}{\det\left(g_{AB}\right)}}\right]^{\frac{1}{n+2}}\left(y^A+h^{AB}\mathbf I_B(y)\right){\frac{\partial}{\partial y^A}}
\end{aligned}
\end{equation}
for $\forall y\in M^n$, where $\triangle_{G}$ is the Beltrami-Laplacian of the Blaschke metric $G$.
Also, the cubic form at $y\in M^n$ induced by the affine norm field $\xi$ is
\begin{equation}\label{eq215}
\mathbf C(y)=2\left[{\frac{\sigma_F^2}{\det\left(g_{AB}\right)}}\right]^{-{\frac{1}{n+2}}}\mathbf M(y),
\end{equation}
where $\mathbf M(y)$ is the Matsumoto torsion defined in (\ref{eq28}). In the case of $n=2$, we recall the definition of the Darbous curve (cf. \cite{Su ADG}, \S{I.5}):
\begin{Definition}\label{Darboux}
A curve $\gamma\subset M^2$ is called a \emph{Darboux curve} if $\gamma$ is the integral curve of the mull direction of $\mathbf{C}$.
\end{Definition}
The shape operator with respect to the equiaffine structrue is
\begin{equation}\label{eq216}
s_i^j=\left(\delta_i^j+{\frac{2}{n+2}}h^{jk}\mathbf I_{i;k}-{\frac{2n}{(n+2)^2}}h^{jk}\mathbf I_i\mathbf I_k\right),
\end{equation}
where $\mathbf I_i=\mathbf I_A{\frac{\partial r^A}{\partial \theta^i}}$ and ";" is the covariant derivative with respect to the Levi-Civita connection of the angular form $h$.
Finally, let $\{\lambda_1,\,...\,,\lambda_n\}$ be $n$ eigenvalues of $s_i^j$, they are the affine principle curvatures of $M^n$. For $1\leq k\leq n$, we define here
\begin{equation}\label{eq217}
L_k:={\frac{k!(n-k)!}{n!}}\sum_{i_1<...<i_k}\lambda_{i_1}...\lambda_{i_k}
\end{equation}
the k-th affine mean curvature of $M^n$.
\\
\section{Characterization of 3-dimensional general $(\alpha,\,\beta)$ norms}
\label{sec:3}
In this section, we will give a characterization of 3-dimensional general $(\alpha,\,\beta)$ norms by studying the induced norm on a certain 2-dimensional subspace.

In Yu-Zhu\cite{Yu-Zhu DGA}, the general $(\alpha,\,\beta)$ norm is investigated. Roughly speaking, it can be derived from an $(\alpha,\,\beta)$ norm by solving the navigation problem (\ref{eq29}). Hence its indicatrix can be derived by shifting the indicatrix of an $(\alpha,\,\beta)$ norm. In \cite{Yu-Zhu DGA}, Yu-Zhu proved that in an $\alpha$-orthogonal coordinate, the indicatrix of an $(\alpha,\,\beta)$ norm is a hypersurface of revolution whose axis passes through the origin. Further, the indicatrix of a general $(\alpha,\,\beta)$ norm is a hypersurface with its axis possibly does not pass the origin.
\begin{Theorem}\label{Thm61a}
(Theorem 2.2 in \cite{Yu-Zhu DGA})
Let $F$ be a Minkowski norm on a vector space $V$ of dimension $n\geq2$. Then $F$ is an $(\alpha,\,\beta)$ norm if and only if $F$ is $G$- invariant, where
$$G=\left\{g\in GL(n,\mathbb{R})\mid g=diag(A,1),\quad A\in O(n-1)\right\}.$$
\end{Theorem}
Theorem \ref{Thm61a} is confusing, because of the uncertainty of coordinates of V, as well as the uncertainty of the group action. Yu-Zhu's proof of the above theorem further showed that if $F$ is $G$-invariant, then $\alpha$ is expressed by
\begin{equation}
\alpha(y)=\sqrt{\sum_{A=1}^n(y^A)^2}
\end{equation}
and $\beta$ by $\beta(y)=by^n$ for some constant $b\in \mathbb{R}$. So their theorem actually only works for $\alpha$-orthogonal coordinates which can not be chosen a priori  in the proof of sufficiency. At least, the proof in \cite{Yu-Zhu DGA} leads to
\begin{Theorem}\label{Thm61}
Let $F$ be a Minkowski norm on a vector space $V$ of dimension $n\geq2$ with $\{y^i\}_{i=1}^n$ a fixed coordinate. Then $F$ is an $(\alpha,\,\beta)$ norm expressed by
\begin{equation}
\begin{aligned}
&F(y)=\alpha\phi\left({\frac{\beta}{\alpha}}\right),\\
&\alpha(y)=\sqrt{\sum_{A=1}^n(y^A)^2},\\
&\beta(y)=by^n,\quad b\in \mathbb{R}
\end{aligned}
\end{equation}
if and only if $F$ is $G$- invariant, where
$$G=\left\{g\in GL(n,\mathbb{R})\mid g=diag(A,1),\quad A\in O(n-1)\right\}.$$
\end{Theorem}
In order to give a characterization of the general $(\alpha,\,\beta)$ norms independent of the choice of coordinates, let's first check the 3-dimensional case. In this case the indicatrix of the Minkowski is a surface. The following affine description of the indicatrix is crucial in this section:

\begin{Proposition}\label{prop61}
Let $V^3$ be a $3$-dimensional Minkowski space and $\bar{F}$ a Minkowski norm on it. Let $\bar{M}^2$ be the indicatrix of $\bar{F}$ in $V^3$, then $\bar{F}$ is a general $(\alpha,\,\beta)$ norm if and only if $\bar{M}^2$ is an affine surface of revolution.
\end{Proposition}

\begin{proof}
$(\Longrightarrow)$ By definition, for some suitable vector $v\in V^3$, $\bar{M}^2+v$ is the indicatrix of some $(\alpha,\,\beta)$ norm. Now it is obviously seen by Theorem \ref{Thm61} that $\bar{M}^2+v$ is an affine surface of revolution, hence for $\bar{M}^2$.\\
$(\Longleftarrow)$ Let $\{\bar{y}^A\}$ be the original coordinate of $V^3$. It was proved in \S{IV.2} of Su \cite{Su ADG} that there are only three types of affine surfaces of revolution, namely parabolic (Type (I)), elliptic (Type (II)) and hyperbolic (Type (III)). As the indicatrix $\bar{M}^2$ is a compact surface, it must be of Type (II), i.e. one branch of the affine curvature curves are parallel curves and they are ellipses. These ellipses also coincide with one branch of Darboux curves. Moreover, in \cite{Su ADG}(pp.120-122), it is proved that for any affine surface of revolution of elliptic type, one can choose an affine coordinate $\{{y^*}^A\}$ (which probably changes the origin of $V^3$) such that the surface $\bar M^2$ is given by
\begin{equation}\label{61}
\left\{
    \begin{aligned}
    {y^*}^1(\theta^1,\,\theta^2)&=\theta^2\\
    {y^*}^2(\theta^1,\,\theta^2)&={\frac{a_2}{a_1}}\exp(-\int(\psi(\theta^2)-\theta^2)^{-1}d\theta^2)\cos(\kappa \theta^1)\\
   {{y^*}^3}(\theta^1,\,\theta^2)&={\frac{a_3}{a_1}}\exp(-\int(\psi(\theta^2)-\theta^2)^{-1}d\theta^2)\sin(\kappa \theta^1)
   \end{aligned}
    \right.
\end{equation}
where $a_1,\,a_2,\,a_3$ and $\kappa>0$ are constants and $\psi:\mathbb{R}\to\mathbb{R}$ is a function. By a translating along the ${y^*}^1$-direction, one can assume that $O^*$, the origin of the coordinate system $\{{y^*}^A\}_{A=1}^3$, is enclosed by $\bar M^n$.

Let
\begin{equation}
\tilde y^1={y^*}^1,\quad \tilde y^2={\frac{{y^*}^2}{a_2}},\quad \tilde y^3={\frac{{y^*}^3}{a_3}},
\end{equation}
now Theorem \ref{Thm61} applies in the affine coordinate $\{\tilde{y}^A\}$. Choose $\{\tilde{y}^A\}$ as the coordinate of $V^3$, and let $\tilde F(\tilde y)$ be the Minkowski norm on $V^3$ whose indicatrix is $\bar M^2$, then Theorem \ref{Thm61} implies that $\tilde{F}(\tilde y)$ is an $(\alpha,\,\beta)$ norm. Let the origin of the coordinate $\{\tilde y^A\}$ be $\tilde O$ and that of $\{\bar{y^A}\}$ be $O$, then $\bar F$ is a Minkowski norm with navigation data $(\tilde F, \overrightarrow{\tilde OO})$. Hence $\bar F$ is a general $(\alpha,\,\beta)$ norm.
\end{proof}

Su had showed in his pioneering works that
\begin{Theorem}\label{Thm62}
(Theorem 22 of \cite{Su Jap 2}, Theorem 36 of \cite{Su Jap 6})
One branch of the Darboux curves of a surface lies on parallel planes if and only if the surface is an affine surface of revolution or an affine sphere of Type (I), (II) or (III) defined in \cite{Su Jap 6}(also, see \cite{Su ADG}, pp.120-122).
\end{Theorem}
Now we are going to character 3-dimensional general $(\alpha,\,\beta)$ norms by using Theorem \ref{Thm62}, as mentioned in Proposition \ref{prop61}, Type (I) and (III) are excluded in our case. First, let's find out that how can a planer curve on $\bar{M}^2$ be its Darboux curve.
Suppose $W^2$ is a 2-dimensional subspace of $V^3$, then $\bar{F}$ induces a Minkowski norm $F$ on $W^2$. From now on, denote objects with respect to $\bar F$ by adding a bar, and corresponding objects of $F$ without it. Without loss of generality, one can assume that
\begin{equation}
\begin{aligned}
W^2&=Span\left\{{\frac{\partial}{\partial y^1}},\,{\frac{\partial}{\partial y^2}}\right\}\\
V^3&=Span\left\{{\frac{\partial}{\partial y^1}},\,{\frac{\partial}{\partial y^2}},\,{\frac{\partial}{\partial y^3}}\right\}.
\end{aligned}
\end{equation}
By definition, the indicatrix of $F$, denoted by $M^1$, is
\begin{equation}\label{eq64}
M^1=\bar{M}^2\cap W^2.
\end{equation}
which is a strongly convex closed curve lies in the 2-plane $V$.
The fundamental tensor of $F$ is
\begin{equation}\label{eq61}
g_{ij}={\frac{1}{2}}{\frac{\partial^2F}{\partial y^i\partial y^j}}={\frac{1}{2}}{\frac{\partial^2\bar F}{\partial y^i\partial y^j}}=\bar g_{ij},
\end{equation}
the angular form of $F$ is
\begin{equation}\label{eq62}
h_{ij}=g_{ij}-{\frac{\partial F}{\partial y^i}}{\frac{\partial F}{\partial y^j}}=\bar h_{ij},
\end{equation}
and the Cartan tensor of $F$ is
\begin{equation}\label{eq63}
\mathbf A_{ijk}=F{\frac{1}{2}}{\frac{\partial g_{ij}}{\partial y^k}}=\bar F{\frac{1}{2}}{\frac{\partial \bar g_{ij}}{\partial y^k}}=\bar{\mathbf A}_{ijk},
\end{equation}
where $1\leq i,\,j,\,k\leq2$.

Nearby $M^1$, one can choose a local coordinate $(\theta^1,\,\theta^2)$ of $\bar M^2$ such that $M^1$ is represented by $\theta^2=0$. This can be done since (i) $M^1\subset W^2$ is diffeomorphic to the standard circle $S^1$, so it can be parameterized as
\begin{equation}
\gamma: S^1\left(\cong [0,1]/\{0,1\}\right)\to W^2\left(\subset V^3\right)
\end{equation}
\begin{equation}
\left\{
\begin{aligned}
\gamma^1(\theta^1)&=\rho(\theta^1)\cos \theta^1\\
\gamma^2(\theta^1)&=\rho(\theta^1)\sin \theta^1\\
\gamma^3(\theta^1)&=0
\end{aligned}
\right.
\end{equation}
where $\rho(\theta^1)$ is a $C^{\infty}$ function on $S^1$; and (ii) $\bar M^2$ is transverse to $W^2$, so one can take $\theta^2=y^3$ nearby $M^1$.

Along $M^1$, the Matsumoto torsion of $\bar M^2$ is
\begin{equation}\label{eq65}
\begin{aligned}
&\bar{\mathbf M}\left(\gamma(\theta^1)\right)\left({\frac{\partial \gamma}{\partial \theta^1}},{\frac{\partial \gamma}{\partial \theta^1}},{\frac{\partial \gamma}{\partial \theta^1}}\right)
\\&=
\bar{\mathbf A}\left(\gamma(\theta^1)\right)\left({\frac{\partial \gamma}{\partial \theta^1}},{\frac{\partial \gamma}{\partial \theta^1}},{\frac{\partial \gamma}{\partial \theta^1}}\right)
-{\frac{3}{4}}\bar{\mathbf I}\left(\gamma(\theta^1)\right)\left({\frac{\partial \gamma}{\partial \theta^1}}\right)\bar h\left(\gamma(\theta^1)\right)\left({\frac{\partial \gamma}{\partial \theta^1}},{\frac{\partial \gamma}{\partial \theta^1}}\right)
\\&=
\mathbf A\left(\gamma(\theta^1)\right)\left({\frac{\partial \gamma}{\partial \theta^1}},{\frac{\partial \gamma}{\partial \theta^1}},{\frac{\partial \gamma}{\partial \theta^1}}\right)
-{\frac{3}{4}}\bar{\mathbf I}\left(\gamma(\theta^1)\right)\left({\frac{\partial \gamma}{\partial \theta^1}}\right)h\left(\gamma(\theta^1)\right)\left({\frac{\partial \gamma}{\partial \theta^1}},{\frac{\partial \gamma}{\partial \theta^1}}\right)
\\&=
\mathbf I\left(\gamma(\theta^1)\right)\left({\frac{\partial \gamma}{\partial \theta^1}}\right)h\left(\gamma(\theta^1)\right)\left({\frac{\partial \gamma}{\partial \theta^1}},{\frac{\partial \gamma}{\partial \theta^1}}\right)-{\frac{3}{4}}\bar{\mathbf I}\left(\gamma(\theta^1)\right)\left({\frac{\partial \gamma}{\partial \theta^1}}\right)h\left(\gamma(\theta^1)\right)\left({\frac{\partial \gamma}{\partial \theta^1}},{\frac{\partial \gamma}{\partial \theta^1}}\right)
\\&=
\left[\mathbf I\left(\gamma(\theta^1)\right)-{\frac{3}{4}}\bar{\mathbf I}\left(\gamma(\theta^1)\right) \right]\left({\frac{\partial \gamma}{\partial \theta^1}}\right)h\left(\gamma(\theta^1)\right)\left({\frac{\partial \gamma}{\partial \theta^1}},{\frac{\partial \gamma}{\partial \theta^1}}\right)
\end{aligned}
\end{equation}
where the second inequality holds by (\ref{eq61})-(\ref{eq63}) and the third comes from the fact that the Matsumoto torsion of a 2-dimensional Minkowski norm always vanishes (cf. \cite{Shenyibing} \S{2.2.2}).

Recall (\ref{eq215}), we see that $M^1$ (or equivalently $\gamma(\theta^1)$) is a Darboux curve of $\bar M^2$ if and only if $\bar{\mathbf M}\left(\gamma(\theta^1)\right)\left({\frac{\partial \gamma}{\partial \theta^1}},{\frac{\partial \gamma}{\partial \theta^1}},{\frac{\partial \gamma}{\partial \theta^1}}\right)$ always vanishes, by (\ref{eq65}), we have
\begin{equation}\label{eq66}
\begin{aligned}
&\bar{\mathbf M}\left(\gamma(\theta^1)\right)\left({\frac{\partial \gamma}{\partial \theta^1}},{\frac{\partial \gamma}{\partial \theta^1}},{\frac{\partial \gamma}{\partial \theta^1}}\right)\equiv0 \\&\Longleftrightarrow
\mathbf I\left(\gamma(\theta^1)\right)-{\frac{3}{4}}\bar{\mathbf I}\left(\gamma(\theta^1)\right)\equiv0\\
&\Longleftrightarrow
{\frac{d}{d \theta^1}}\log\left[{\frac{{\det (g_{ij})}^{\frac{1}{3}}}{{\det (\bar g_{AB})}^{\frac{1}{4}}}}\left(\gamma(\theta^1)\right)\right]\equiv0\\
&\Longleftrightarrow
\log\left[{\frac{{\det (g_{ij})}^{\frac{1}{3}}/\sigma_F^{\frac{2}{3}}}{{\det (\bar g_{AB})}^{\frac{1}{4}}/\sigma_{\bar F}^{\frac{1}{2}}}}\left(\gamma(\theta^1)\right)\right]\equiv constant.
\end{aligned}
\end{equation}
Which can be rewritten in terms of distorsions as
\begin{equation}\label{eq67}
{\frac{1}{3}}\tau\left(\gamma(\theta^1)\right)-{\frac{1}{4}}\bar\tau\left(\gamma(\theta^1)\right)\equiv constant.
\end{equation}
We summarize the above discussions in the following
\begin{Proposition}\label{prop62}
$M^1$ is the Doarboux curve of $\bar M^2$ if and only if
\begin{equation}
\mathbf T:={\frac{1}{3}}\tau-{\frac{1}{4}}\bar\tau
\end{equation}
is a constant along $M^1$.
\end{Proposition}
We are going to prove the main theorem of this section, before this, let's agree with some notations.
Suppose $U$ is a vector in $V^3$ and
\begin{equation}\label{eq68}
\bar F(U)<1
\end{equation}
and let $\bar F_U(y)$ be the Minkowski norm with navigation data $(\bar F, U)$ (cf. \ref{eq29}), then the indicatrix of $\bar F_U$, denoted by $\bar M^2_U$, equals to $\bar M^2-U$ as point sets.

Denote $F_U$ the norm on $W^2$ induced by $\bar F_U$, with its idicatrix denoted by $M_U^1$, then
\begin{equation}\label{eq610}
M^1_U=\bar M^2_U\cap W^2
\end{equation}
which is a strictly convex closed curve on $W^2$.

For each $p\in \bar M^2$ (except two points at which the tangent plane of $\bar M^2$ is parallel to $W^2$), $\bar M^2\cap \left\{W^2+\overrightarrow{Op}\right\}$ is a strictly convex closed curve parallel to $W^2$. On the other hand, for any planer curve $\Gamma$ on $\bar M^2$ parallel to $W^2$, one can choose a vector $U$ such that
\begin{equation}
\Gamma-U=M^1_U,
\end{equation}
for example, let $\Gamma=\bar M^2\cap\left\{W^2+U'\right\}$ for some vector $U'\in V^3$, take a point $q$ in the domain bounded by $\Gamma$ in $\left\{W^2+U'\right\}$, then we can choose $U=\overrightarrow{Oq}$ since $F(\overrightarrow{Oq})<1$.

Let's state and prove the following:
\begin{Theorem}\label{Thm63}
Suppose $\bar F$ is a Minkowski norm on a 3-dimensional vector space $V^3$. Then $\bar F$ is a general $(\alpha,\,\beta)$ norm if and only if
\begin{itemize}
\item [(*)] There exists a 2-dimensional subspace $W^2$ of $V^3$ such that for any $\bar F_U$ determined by (\ref{eq29}) and (\ref{eq68}), the norm $F_U$ induced by $\bar F_U$ on $W^2$ satisfies:
    \begin{equation}\label{eq611}
    {\frac{1}{3}}\tau_U-{\frac{1}{4}}\bar\tau_U= constant
    \end{equation}
\end{itemize}
on $W^2\backslash \{O\}$ ,where $\tau_U$ and $\bar\tau_U$ are the distorsions with respect to $F_U$ and $\bar F_U$.
Furthermore, if $\bar F$ is not a Randers norm and the condition $(*)$ is satiesfied, then
\begin{itemize}
\item [(i)] $W^2$ is defined by $\beta=0$;
\item [(ii)] $F_U$'s are Randers norms.
\end{itemize}
\end{Theorem}

\begin{proof}
$(\Longrightarrow)$ Suppose $\bar F$ is a general $(\alpha,\,\beta)$ norm which is not of Randers type, one can choose a coordinate $\{y^A\}_{i=1}^{3}$ on $V^3$ such that
\begin{equation}
\left\{\begin{aligned}\alpha(y)&=\sqrt{\sum_{A=1}^{3}(y^A)^2}\\
\beta(y)&=by^3,\,b=constant.
\end{aligned}\right.
\end{equation}
Take $W^2=Span\left\{{\frac{\partial}{\partial y^1}},{\frac{\partial}{\partial y^2}}\right\}$, then $(*)$ follows immediately. For $(i)$ and $(ii)$, Since Proposition \ref{prop61} implies that $\bar M^2$ is an affine surface of revolution, it sufficient to show that when $\bar F$ is not a Randers norm, the axis of $\bar M^2$ is uniquely determined. One can check that (see \cite{Su ADG}, \S{IV.1} and \S{IV.2}):
 \begin{itemize}
 \item [(a)] The $y^3$-axis is parallel to the axis of $\bar M^2$ (denoted by $L$);
 \item [(b)] All curves on $\bar M^2$ parallel to $W^2$ are planer affine lines of curvature;
 \item [(c)] All the meridian curves of $M^2$ are affine lines of curvature, and each meridian line intersects with $L$ at exactly two points.
 \end{itemize}
 If their is another 2-dimensional subspace $W'^2\not=W^2$ satisfies $(*)$, then one can get another branch of planer curvature lines parallel to $W'^2$. So one will get at least three different branches affine lines of curvature on $\bar M^2$, this forces $\bar M^2$ to be an affine sphere, hence $\bar F$ is a Randers norm, a contradiction.\\
 $(\Longleftarrow)$ Suppose that $(*)$ is held, for each $p\in \bar M^2$, take $\Gamma_p=\bar M^2\cap\{W^2+\overrightarrow{Op}\}$. By the discussion before this theorem, we can represent $\Gamma_p$ as $M^1_U+U$ for some $U\in V^3$. While (*) and Proposition \ref{prop62} implies that $\Gamma_p-U$ is a Darboux curve of $\bar M^2_U$, then $\Gamma_p$ is a Darboux curve of $\bar M^2$ passes through $p$. Now we've found one branch of Darboux curves on $\bar M^2$ parallel to $W^2$, by Theorem \ref{Thm62}, $\bar M^2$ is an affine surface of revolution, hence $\bar F$ is a general $(\alpha,\,\beta)$ norm by Proposition \ref{prop61}.
\end{proof}

\section{Some global characterizations of Randers norms}
\label{sec:4}
In this section, we will derive some global geometric quantities which characterizes Randers norms of arbitrary dimensions. The ideal is using the affine isoperimetric inequalities to characterize affine hyperspheres.

We define the affine volume of the indicatrix as the following:
\begin{Definition}\label{df51}
Let $V^{n+1}$ be an $(n+1)$-dimensional vector space and $F$ a Minkowski norm on it. Let $dV_{B-H}$ the Busemann-Hausdorff volume form on $V^{n+1}$. Denote the indicatrix of $F$ by $M^n$. The affine volume of $M^n$ is
\begin{equation}\label{eq50}
S(M^n):=\int_{M^n}1\xi\lrcorner dV_{B-H}.
\end{equation}
where $\xi$ is the affine norm field on $M^n$ defined in (\ref{eq214}).
\end{Definition}
Let's first take a look at Randers norms. For a Randers norm $F_R$, let $(G,W)$ be its navigation data, i.e. $G(y)=\sqrt{G_{AB}y^Ay^B}$ is a Euclidean norm and $W=w^A{\frac{\partial}{\partial y^A}}$ is a vector with ${\|W\|}_{G}^2:=G_{AB}W^AW^B<1$. Plugging $G$ and $W$ into (\ref{eq29}) and by a direct computation, we have
\begin{equation}\label{eqr1}
F_R(y)=\alpha(y)+\beta(y),
\end{equation}
with
\begin{equation}\label{eqr2}
\begin{aligned}
&\alpha(y)=\sqrt{\left({\frac{G_{AB}}{\lambda}}+{\frac{G_{AC}G_{BD}W^CW^D}{\lambda^2}}\right)y^Ay^B}, \\
&\beta(y)={\frac{G_{AB}W^Ay^B}{\lambda}},\\
&\lambda=1-{\|W\|}_{G}^2.
\end{aligned}
\end{equation}
Thus, we have by \cite{Chern and Shen 2005} \S{1.1}
\begin{equation}\label{eqr3}
\begin{aligned}
&\det\left(g_{AB}\right)=\left({\frac{F_R(y)}{\lambda\alpha(y)}}\right)^{n+1}\det\left(G_{AB}\right),\\
&\mathbf I^A(y)={\frac{(n+1)\alpha(y)}{2F_R^2(y)}}\left(\lambda W^A-{\frac{y^A\beta(y)}{\alpha^2(y)}}-{\frac{(1-\lambda)y^A}{F_R(y)}}+{\frac{\beta^2(y)y^A}{\alpha^2(y)F_R(y)}} \right),\\
&\sigma_{F_R}=\sqrt{\det\left(G_{AB}\right)}.
\end{aligned}
\end{equation}
By restricting the above (\ref{eqr1})-(\ref{eqr3}) on the indicatrix of $F_R(y)$ (denoted by $M_R^n$) and combining (\ref{eq21})-(\ref{eq27}) and (\ref{eq214}), we get the affine norm of $M_R^n$
\begin{equation}
\xi(y)=-\left(y^A+W^A\right){\frac{\partial}{\partial y^A}}.
\end{equation}
On the other hand, (\ref{eq210}) and (\ref{eq68}) imply that
\begin{equation}\label{eqr4}
M_R^n=\left\{y\in V^{n+1}\mid G_{AB}y^Ay^B<1\right\}-W
\end{equation}
as point set. Hence $M_R^n$ is an affine hypersphere centred at $-W$ with constant radii $1$ in $V^{n+1}$. It is also easy to check that $S(M_R^n)=(n+1)\omega_{n+1}$ and $L_r=1$ for all $r=1,\,...\,,n$. The case of $r=1$ then implies that if $L_1$ in Theorem \ref{Thm13} is a constant along the indicatrix, then it must be precisely $1$.

For an arbitrary Minkowski norm, we have the following:
\begin{Theorem}\label{Thm52}
\begin{equation}\label{eq50a}
S(M^n)\leq(n+1)\omega_{n+1}
\end{equation}
with the equality holds if and only if $F$ is a Randers metric.
\end{Theorem}
\begin{proof}
Choose an inner product $\langle-,-\rangle$ on $V^{n+1}$ such that $\langle{\frac{\partial}{\partial y^A}},{\frac{\partial}{\partial y^B}}\rangle=\delta_{AB}$.
Choose $\{\theta^i\}_{i=1}^{n}$ as the parameter of $M^n$, therefore the Blaschke metric on $M^n$ is
\begin{equation}\label{eq51}
G=G_{ij}d\theta^i\otimes d\theta^j,\,G_{ij}=\left[{\frac{\sigma_F^2}{\det\left(g_{AB}\right)}}\right]^{\frac{1}{n+2}}h_{ij}.
\end{equation}
We obtain the affine volume element of $M^n$ as
\begin{equation}\label{eq52}
\begin{aligned}
\xi\lrcorner dV_{B-H}&=\sqrt{\det\left(G_{ij}\right)}d\theta^1\wedge...\wedge d\theta^n\\
&={\frac{\sigma_F^{\frac{n}{n+2}}}{{\det\left(g_{AB}\right)}^{\frac{n}{n+2}}}}\sqrt{\det\left(h_{ij}\right)}d\theta^1\wedge...\wedge d\theta^n.
\end{aligned}
\end{equation}
By (\ref{eq311}),
\begin{equation}
\det\left(h_{ij}\right)=\det\left(g_{AB}\right)\left[\det\left({\frac{\partial r}{\partial \theta}}\mid r\right)\right]^2,
\end{equation}
where
\begin{equation}
\left({\frac{\partial r}{\partial \theta}}\mid r\right)=
\left({\frac{\partial r}{\partial \theta^1}},\,...\,,{\frac{\partial r}{\partial \theta^n}},r
\right).
\end{equation}
so we have
\begin{equation}\label{eq53}
\begin{aligned}
&\sqrt{\det\left(G_{ij}\right)}d\theta^1\wedge...\wedge d\theta^n\\
&={\sigma_F^{\frac{n}{n+2}}}{\sqrt{\det\left(h_{ij}\right)}}^{\frac{2}{n+2}}\left[\det\left({\frac{\partial r}{\partial \theta}}\mid r\right)\right]^{\frac{n}{n+2}}d\theta^1\wedge...\wedge d\theta^n.
\end{aligned}
\end{equation}
Following Blaschke \cite{Blaschke ADG} and Li-Zhao \cite{Li-Zhao ADG}, let's transform the affine volume form of $M^n$ into a "Euclidean" one. Denote $dV_E$ the volume form of $M^n$ induced by the inner product $\langle-,-\rangle$, it is easy to check that
\begin{equation}\label{eq55}
dV_E=\sqrt{\det\left({\frac{\partial r}{\partial \theta}}\mid \nu\right)^T\left({\frac{\partial r}{\partial \theta}}\mid \nu\right)}^{\frac{1}{2}}d\theta^1\wedge...\wedge d\theta^n,
\end{equation}
where
\begin{equation}
\left({\frac{\partial r}{\partial \theta}}\mid \nu\right)=
\left({\frac{\partial r}{\partial \theta^1}},\,...\,,{\frac{\partial r}{\partial \theta^n}},\nu
\right)
\end{equation}
and $\nu$ is the unit normal field of $M^n$ with respect to $\langle-,-\rangle$. Set
$$\textrm{II}=\textrm{II}_{ij}d\theta^i\otimes d\theta^j$$ the second fundamental form of $M^n$ with respect to $\langle-,-\rangle$ by
\begin{equation}\label{eq56}
D_{\frac{\partial r}{\partial \theta^i}}{\frac{\partial r}{\partial \theta^j}}={\Gamma}^{*k}_{ij}{\frac{\partial r}{\partial \theta^k}}-\textrm{II}_{ij}\nu.
\end{equation}
We have by (\ref{59plus}) and (\ref{eq56}) that
\begin{equation}\label{eq58}
\begin{aligned}
h_{ij}&={\frac{-\det\left({\frac{\partial r}{\partial \theta^1}},...,{\frac{\partial r}{\partial \theta^n}},D_{\frac{\partial r}{\partial \theta^i}}{\frac{\partial r}{\partial \theta^j}}\right)}{\det\left({\frac{\partial r}{\partial \theta^1}},...,{\frac{\partial r}{\partial \theta^n}},r\right)}}\\
&={\frac{\textrm{II}_{ij}\det\left({\frac{\partial r}{\partial \theta}}\mid \nu\right)}{\det\left({\frac{\partial r}{\partial \theta^1}},...,{\frac{\partial r}{\partial \theta^n}},r\right)}}
\end{aligned}
\end{equation}
Plugging (\ref{eq58}) into (\ref{eq53}), we obtain
\begin{equation}\label{eq59}
\begin{aligned}
\sqrt{\det\left(G_{ij}\right)}d\theta^1\wedge...\wedge d\theta^n&=\sigma_F^{\frac{n}{n+2}}\left[\det\left(\textrm{II}_{ij}\right)\det\left({\frac{\partial r}{\partial \theta}}\mid \nu\right)^n\right]^{\frac{1}{n+2}}d\theta^1\wedge...\wedge d\theta^n\\
&=\sigma_F^{\frac{n}{n+2}}K^{\frac{1}{n+2}}\mid \det\left({\frac{\partial r}{\partial \theta}}\mid \nu\right)\mid d\theta^1\wedge...\wedge d\theta^n\\
&=\sigma_F^{\frac{n}{n+2}}K^{\frac{1}{n+2}}dV_E
\end{aligned}
\end{equation}
where $K={\frac{\det\left(\textrm{II}_{ij}\right)}{\det\left({\frac{\partial r}{\partial \theta}}\mid \nu\right)^T\left({\frac{\partial r}{\partial \theta}}\mid \nu\right)}}$ is the Gauss-Kronecker curvature of $M^n$. By (\ref{eq58}), we have
\begin{equation}\label{eq59a}
K=\det(h_{ij}){\frac{\left[\det\left({\frac{\partial r}{\partial \theta^1}},...,{\frac{\partial r}{\partial \theta^n}},r\right)\right]^n}{\left[\det\left({\frac{\partial r}{\partial \theta}}\mid \nu\right)\right]^{n+2}}}.
\end{equation}
We can conclude that $K>0$ everywhere on $M^n$, since:
\begin{itemize}
 \item[(i)] The angular form $(h_{ij})$ is strictly positive definite;
 \item[(ii)] The origin $O$ locates strictly in the interior of the domain enclosed by $M^n$ and $\nu$ is the outer norm, hence
 \begin{equation}\label{eq59b}
 {\frac{\det\left({\frac{\partial r}{\partial \theta^1}},...,{\frac{\partial r}{\partial \theta^n}},r\right)}{\det\left({\frac{\partial r}{\partial \theta}}\mid \nu\right)}}=\langle r,\,\nu\rangle>0.
 \end{equation}
\end{itemize}
Hence $M^n$ is an ovaloid with respect to the chosen inner product, and the classical theory of ovaloids may then apply to $M^n$.
By the above computations, we obtain
\begin{equation}\label{eq510}
S(M^n)=\sigma_F^{\frac{n}{n+2}}\int_{M^n}K^{\frac{1}{n+2}}dV_E
\end{equation}
To estimate RHS of (\ref{eq510}), we use the H\"{o}lder inequality and the convexity of $M^n$, precisely we have
\begin{equation}\label{eq511}
\left({\frac{\int_{M^n}K^{\frac{1}{n+2}}dV_E}{\int_{M^n}1dV_E}}\right)^{n+2}\leq{\frac{\int_{M^n}KdV_E}{\int_{M^n}1dV_E}}={\frac{(n+1)\omega_{n+1}}{\int_{M^n}1dV_E}}
\end{equation}
Hence, we have
\begin{equation}\label{eq511a}
S(M^n)\leq \sigma_F^{\frac{n}{n+2}}\left[(n+1)\omega_{n+1}\left(\int_{M^n}1dV_E\right)^{n+1}\right]^{\frac{1}{n+2}}
\end{equation}
Set a standard Euclidean ball $\mathbf{B}$ (with $\Sigma^n$ its boundary) of volume $Vol_{\mathbb{R}^{n+1}}(\mathbf{B})=Vol_{\mathbb{R}^{n+1}}(B_F(1))$, we have the $n$-dimensional volume of $\Sigma^n$ is
\begin{equation}\label{eq512}
S_E(\Sigma^n):=\int_{\Sigma^n}1dV_E=(n+1)\omega_{n+1}^{\frac{1}{n+1}}Vol_{\mathbb{R}^{n+1}}^{\frac{n}{n+1}}(\mathbf{B}).
\end{equation}
Due to the convergency theorem of W.Gross (cf. Blaschke \cite{Blaschke ADG}, \S{115}), for $\forall \epsilon>0$, by taking sufficiently many suitable Steiner symmetrization of $M^n$, one can construct a new ovaloid $\dot{M}^n$ such that
\begin{equation}\label{eq512a}
S_E(\dot{M}^n):=\int_{\dot{M}^n}1dV_E<S_E(\Sigma^n)+\epsilon
\end{equation}
and
\begin{equation}\label{eq512b}
\begin{aligned}
S(M^n)\leq S(\dot{M}^n)&\leq \sigma_F^{\frac{n}{n+2}}\left[(n+1)\omega_{n+1}S_E(\dot{M}^n)^{n+1}\right]^{\frac{1}{n+2}}\\
&\leq \sigma_F^{\frac{n}{n+2}}\left[(n+1)\omega_{n+1}(S_E(\Sigma^n)+\epsilon)^{n+1}\right]^{\frac{1}{n+2}},
\end{aligned}
\end{equation}
where we've used the analogous inequality of (\ref{eq511a}) for $\dot M^n$. Take $\epsilon\to0$, we have
\begin{equation}\label{eq513}
S(M^n)\leq(n+1)\sigma_F^{\frac{n}{n+2}}\left[\omega_{n+1}^2Vol_{\mathbb{R}^{n+1}}^{n}(B_F(1))\right]^{\frac{1}{n+2}}.
\end{equation}
By definition of the Busemann-Hausdorff volume form, we have
\begin{equation}
\sigma_F Vol_{\mathbb{R}^{n+1}}(B_F(1))=\omega_{n+1}.
\end{equation}
Plugging the above equation into (\ref{eq513}), (\ref{eq50a}) is proved.

Suppose $S(M^n)=(n+1)\omega_{n+1}$ for some $M^n$, then the midpoints of parallel chords along any direction must lie on an $n$-dimensional hyperplane, hence $M^n$ must be an ellipsoid (cf. \cite{Li-Zhao ADG} \S{5.1}).
\end{proof}

Theorem \ref{Thm52} is actually an analog of the classical affine isoperimetric inequalities (see Su \cite{Su ADG} \S{II.5}, Blaschke \cite{Blaschke ADG} \S{65}, \S{72}, \S{73}, Li-Zhao \cite{Li-Zhao ADG} \S{5.1}, etc) in an arbitrary Minkowski space. 

\begin{Remark}\label{rmk51}
For an arbitrary volume form $d\bar V=
\bar{\sigma}_F dy^1\wedge...\wedge dy^{n+1}$ defined on $V^{n+1}$, denote
\begin{equation}
Vol_{\bar{\sigma}_F}(\Omega):=\int_\Omega{1d\bar V}
\end{equation}
the F-volume for a measurable set $\Omega$ with respect to $d\bar V$.
Then by a similar argument, Theorem (\ref{Thm52}) is still available in the sense that
\begin{equation}\label{eq514}
\bar S(M^n)\leq(n+1)\left[\omega^2_{n+1}Vol_{\bar \sigma_F}^n(B_F(1))\right]^{\frac{1}{n+2}}
\end{equation}
with $S(M^n):=\int_{M^n}\bar\xi\lrcorner d\bar V$ the corresponding affine area of $M^n$. The equality holds if and only if $F$ is a Randers metric.
\end{Remark}

The above (\ref{eq513}) (or (\ref{eq514}) equivalently) actually leads to the following integral inequalities of affine mean curvatures (Theorem \ref{Thm53} and Theorem \ref{Thm54} below), which can again give characterizations of Randers norms. Because the proofs are similar to those of Theorem 2.3 and Theorem 2.4 in Chapter 5 of \cite{Li-Zhao ADG}, we will just give sketches.

\begin{Theorem}\label{Thm53}
For any integers $k$ and $k^*$ satisfy $0\leq k<k^*\leq n+1,\, k<{\frac{n+2}{2}}$, we have
\begin{equation}\label{eq515}
\left(\int_{M^n}L_{k^*-1}\xi\lrcorner dV_{B-H}\right)^{n+2-2k}\left(\int_{M^n}L_{k-1}\xi\lrcorner dV_{B-H}\right)^{2k^*-n-2}\leq\left((n+1)\omega_{n+1}\right)^{2(k^*-k)}.
\end{equation}
The equality holds if and only if $F$ is a Randers norm. The $L_i$'s are defined in (\ref{eq217}).
\end{Theorem}
\begin{proof}
\emph{Step 1.} We first construct a new hypersurface $\Theta^n$ in $V^{n+1}$ by
\begin{equation}\label{eq516}
\begin{aligned}
\xi:\,M^n&\to\Theta^n\\
y&\to-\xi(y).
\end{aligned}
\end{equation}
As $M^n$ is strictly convex, $L_n>0$ everywhere on $M^n$, hence $\Theta^n$ is diffeomorphic to $M^n$ and is an ovaloid. Denote $\Xi$ the convex domain enclosed by $\Theta^n$. Let $M^n_t$ be a series of hypersurfaces defined by
\begin{equation}\label{eq517}
\begin{aligned}
r_t:\,M^n&\to M_t^n\\
y&\to y-t\xi(y),
\end{aligned}
\end{equation}
and $\Omega_t$ the domain enclosed by $M^n_t$. Note that $\Omega_0=B_F(1)$. Denote the mixed volume (cf. \cite{Burago}, pp.136-138) by
\begin{equation}\label{eq517}
V_k:=Vol_{\sigma_F}\left(\underbrace{B_F(1),\,...\,,B_F(1)}_{(n+1-k) -times},\underbrace{\Xi,\,...\,,\Xi}_{k -times}\right).
\end{equation}
\emph{Step 2.} It can be showed that $M^n_t$'s are all ovaloids, hence $\Omega_t$'s are convex bodies. By $\partial \Omega_t=M^n_t$, we have
\begin{equation}\label{eq518}
Vol_{\sigma_F}(\Omega_t)={\frac{1}{(n+1)!}}\int_{S^n(1)}\det\left({\frac{\partial r_t}{\partial \theta^1}},\,...\,,{\frac{\partial r_t}{\partial \theta^n}},\,r_t\right)d\theta^1\wedge...\wedge d\theta^n,
\end{equation}
where
\begin{equation}\label{eq518a}
r_t(\theta)=r(\theta)-t\xi\left(r(\theta)\right).
\end{equation}
On the other hand, since $\Omega_t$'s are convex bodies, we have
\begin{equation}\label{eq519}
\Omega_t=\Omega_0+t\Xi,
\end{equation}
hence
\begin{equation}\label{eq519a}
Vol_{\sigma_F}(\Omega_t)=\sum_{k=0}^n{\frac{(n+1)!}{k!(n+1-k)!}}V_k.
\end{equation}
Comparing the coefficients of $t^k$ and combining (\ref{eq216}) (\ref{eq217}), we have
\begin{equation}\label{eq520}
\begin{aligned}
&V_{k+1}={\frac{1}{n+1}}\int_{M^n}L_k\xi\lrcorner dV_{B-H},\quad k\geq1\\
&V_1={\frac{1}{n+1}}S(M^n),\quad V_0=Vol_{\sigma_F}\left(B_F(1)\right),
\end{aligned}
\end{equation}
where we've used that
\begin{equation}
{\frac{\partial \xi}{\partial \theta^i}}=-s^j_i{\frac{\partial r}{\partial \theta^j}}
\end{equation}
in computing the first equality of (\ref{eq520}).\\
\emph{Step 3.} The Alexanderov-Fenchel inequality (cf. \cite{Burago}, pp.143) yields that
\begin{equation}\label{eq521}
V_k^2\geq V_{k-1}V_{k+1},\quad1\leq k\leq n
\end{equation}
and recall (\ref{eq513}) that
\begin{equation}\label{eq522}
V_1^{n+2}\leq \omega_{n+1}^2V_0^n.
\end{equation}
Iterating (\ref{eq521}) and (\ref{eq522}) yields
\begin{equation}\label{eq523}
V_k^{2k^*-n-2}V_{k^*}^{n+2-2k}\leq \omega_{n+1}^{2(k^*-k)},
\end{equation}
and combining (\ref{eq520}), (\ref{eq515}) is proved. While the equality holds in (\ref{eq515}) if and only if the equalities hold in (\ref{eq521}) and (\ref{eq522}), hence in (\ref{eq513}) as well, which implies that $M^n$ is an ellipsoid, and hence $F$ is a Randers norm.
\end{proof}
The final characterization is given by integral of $\sqrt{L_n}$:
\begin{Theorem}\label{Thm54}
\begin{equation}\label{eq524}
\int_{M^n}\sqrt{L_n}\xi\lrcorner dV_{B-H}\leq(n+1)\omega_{n+1}.
\end{equation}
where
\begin{equation}\label{eq525}
L_n=\det\left(\delta_i^j+{\frac{2}{n+2}}h^{jk}\mathbf I_{i;k}-{\frac{2n}{(n+2)^2}}h^{jk}\mathbf I_i\mathbf I_k\right),
\end{equation}
as defined in (\ref{eq217}). The equality holds if and only if $F$ is a Randers metric.
\end{Theorem}
\begin{proof}
As $L_n>0$ everywhere on $M^n$, (\ref{eq524}) is a consequence of the H\"{o}lder inequality and (\ref{eq523}). The proof is about the same as that of Theorem 2.4 of \cite{Li-Zhao ADG}, Chapter 5, hence omitted.
\end{proof}
\begin{Remark}
Note that by Remark \ref{rmk51}, Theorem \ref{Thm53} and Theorem \ref{Thm54} hold independently with the choice of $\sigma_F$, hence available for any volume form on $V^{n+1}$.
\end{Remark}

\section{Acknowledgement}
I would like to sincerely thank Zhenye Li for helpful discussions on the manuscript.

%
%
%
%

\begin{thebibliography}{}
%
%


\bibitem{Anosov} Anosov, D. V.: \emph{Geodesics and Finsler geometry.} (in Russian) Proceedings of the International Congress of Mathematicians (Vancouver, B. C., 1974), Vol. 2, 293-297. Canad. Math. Congress, Montreal, Que. (1975).

\bibitem{Bangert} Bangert, V., Long, Y.:  \emph{The existence of two closed geodesics on every Finsler 2-sphere.} Math. Ann. {\bf 346}, 335-366 (2010).

\bibitem{Blaschke ADG} Blaschke, W.: \emph{Vorlesungen \"{u}ber Differentialgeometrie und geometrische Grundlagen von Einsteins Relativit\"{a}tstheorie. Band II. Affine Differentialgeometrie.} Springer-Verlag, Berlin (1923).

\bibitem{GTM200} Bao. D, Chern, S.-S., Shen, Z.: \emph{An introduction to Riemann-Finsler geometry.} Graduate Texts in Mathematics, Vol. 200. Springer-Verlag, New York (2000).

\bibitem{Burago} Burago, Yu.D., Zalgaller, V. A.: \emph{Geometric inequalities.} Grundlehren der Mathematischen Wissenschaften, Vol. 285. Springer Series in Soviet Mathematics. Springer-Verlag, Berlin (1988).

\bibitem{Bryant 2002} Bryant, R.L.: \emph{Some remarks on Finsler
manifolds with constant flag curvature.} Houston J. Math. {\bf 28}, 221-262 (2002).

\bibitem{Chern and Shen 2005} Chern, S.-S., Shen, Z.: \emph{Riemann-Finsler geometry.}
Nankai Tracts in Mathematics, Vol. 6, World Scientific Publishing Co.
Pte. Ltd., Hackensack, NJ (2005).

\bibitem{Deicke} Deicke, A.: \emph{\"{U}ber die Darstellung von Finsler R\"{a}umen durch nichtholonome Mannigfaltigkeiten in Riemannschen R\"{a}umen.} Arch. Math. {\bf 4}, 234-238 (1953).

\bibitem{Jeanne-Chri} Jeanne, N. C., Christopher, G.M.: \emph{Sub-Finsler geometry in dimension three.} Diff. Geom. Appl. {\bf 24}, 628-651 (2006).

\bibitem{Katok} Katok, A. B.: \emph{Ergodic perturbations of degenerate integrable Hamiltonian systems.} (in Russian) Izv. Akad. Nauk SSSR Ser. Mat. {\bf 37}, 539-576 (1973).

\bibitem{Li-Zhao ADG} Li, A., Zhao, G.: \emph{Affine differential geometry.} (in Chinese). SiChuan Education Press, Chengdu (1990).

\bibitem{Matsumoto 1972} Matsumoto, M.: \emph{V-transformations of Finsler spaces. I. Definition, infinitesimal transformations and isometries.} J. Math. Kyoto Univ. {\bf 12}, 479-512 (1972).

\bibitem{Matsumoto 1978} Matsumoto, M., H\={o}j\={o}, S.: \emph{A conclusive theorem on C-reducible Finsler spaces.} Tensor (N.S.) {\bf 32}, 225-230 (1978).

\bibitem{Mo and Huang 2007} Mo, X., Huang, L.: \emph{On curvature decreasing property of a class of navigation problems.} Publ. Math. Debrecen. {\bf 71}, 141-163 (2007).

\bibitem{Mo-Huang 2010} Mo, X., Huang, L.: \emph{On characterizations of Randers norms in a Minkowski space.} Internat. J. Math. {\bf 21}, 523-535 (2010).

\bibitem{Noumizu-Sasaki} Nomizu, K., Sasaki T.: \emph{Affine differential geometry. Geometry of affine immersions.} Cambridge Tracts in Mathematics, Vol. 111. Cambridge University Press, Cambridge (1994).

\bibitem{Randers} Randers, G.: \emph{On an asymmetrical metric in the fourspace of general relativity.} Phys. Rev. (2) {\bf 59}, 195¨C199 (1941).

\bibitem{Shenyibing} Shen, Y., Shen, Z.: \emph{An introduction to modern Finsler geometry.} (in Chinese) Higher Education Press, Beijing (2013).

\bibitem{Su Jap 1} Su, B.: \emph{On the theory of surfaces in the affine space: I. Affine moulding surfaces and affine surfaces of revolution.} Japanese Journal of mathematics {\bf 5}, 185-210 (1928).

\bibitem{Su Jap 2} Su, B.: \emph{On the theory of surfaces in the affine space: II. Generalized affine moulding surfaces and affine surfaces of revolution.} Japanese Journal of mathematics {\bf 5}, 211-224 (1928).

\bibitem{Su Jap 6} Su, B.: \emph{On the theory of surfaces in the affine space: VI. Contributions to the theory of Darboux's curves of the surface.} Japanese Journal of mathematics {\bf 6}, 1-14 (1929).

\bibitem{Su Minkowski} Su, B.: \emph{A characterization of a Euclidean metric as a Minkowski metric.} (in Chinese) Adv. Math. (China). {\bf 7}, 228-230 (1964).

\bibitem{Su ADG} Su, B.: \emph{Affine differential geometry.} Science Press, Beijing; Gordon \& Breach Science Publishers, New York (1983).

\bibitem{Suss} S\"{u}{\ss}, W.: \emph{Ein affingeometrisches Gegenst\"{u}ck zu den Rotationsfl\"{a}chen.} Math. Ann. {\bf 98}, 684-696 (1928).

\bibitem{Wang} Wang, W.: \emph{On a conjecture of Anosov.} Adv. Math. {\bf 230}, 1597-1617 (2012).

\bibitem{Yu-Zhu DGA} Yu, C., Zhu, H.: \emph{On a new class of Finsler metrics.} Diff. Geom. Appl. {\bf 29}, 244-254 (2011).

\end{thebibliography}
%

\end{document}